\documentclass[12]{article}

\usepackage{amsmath}
\usepackage{amsfonts}
\usepackage{latexsym}
\usepackage{amssymb}

\makeatletter

\newtheorem{theorem}{Theorem}[section]
\newtheorem{proposition}[theorem]{Proposition}
\newtheorem{corollary}[theorem]{Corollary}
\newtheorem{lemma}[theorem]{Lemma}

\newtheorem{remark}[theorem]{Remark}

\newcommand{\qed}{\protect\nolinebreak{\hfill$\Box$}}
\newtheorem{unnumber}{}

\newtheorem{prepf}[unnumber]{Proof.}
\newenvironment{proof}{\prepf\rm}\qed{\endprepf}

\def\A{\mathbf A}
\def\a{\mathbf a}

\def\G{\mathbf G}

\def\M{\mathbf M}

\def\X{\mathbf X}
\def\x{\mathbf x}

\def\y{\mathbf y}

\def\z{\mathbf z}
\def\0{\mathbf 0}

\def\cB{\mathcal B}

\def\cH{\mathcal H}

\def\cK{\mathcal K}

\def\cS{\mathcal S}

\def\PG{{\rm PG}}
\def\PGL{{\rm PGL}}

\def\GF{{\rm GF}}
\def\GL{{\rm GL}}

\def\dim{{\rm dim}}

\def\frob{{\rm Fr}}
\def\GF{{\rm GF}}

\newcommand\comment[1]{}

\title{The geometry of elation groups of a finite projective space}

\author{N. Durante and A. Siciliano}

\begin{document}


\maketitle

\begin{abstract}
We study the geometry of point-orbits  of elation groups with a given center and axis of a finite projective space. We show that there exists a 1-1 correspondence from  conjugacy classes of  such groups and orbits on projective subspaces (of a suitable dimension) of Singer groups of  projective spaces. Together with a recent result of Drudge \cite{kd} we establish  the number of these elation groups. 
\end{abstract}
 
\bigskip


\section{Introduction} 
Let $\PG(r-1,p^h)$ be the projective space coordinatized by the finite field $\GF(p^h)$, $p$  a prime and $h$ a positive integer. In  $\PG(r-1,p^h)$, collineations fixing all the points of a hyperplane $A$ are called {\em axial collineations} (or {\em perspectivities}) with {\em axis} $A$. For such collineations there is always a {\em center} $\z$, i.e. a point such that all lines through this point are invariant. A collineation with axis $A$ and center $\z$ is called a $(\z,A)-${\em perspectivity} or a $(\z,A)-${\em central collineation}.   If, for a $(\z,A)-$perspectivity $g$, the image $g(\x)$ of a single point $\x$ different from the center and not lying on the axis is known, then the image of all other points may be obtained by a well-known geometric construction. In particular, if $g$ is not the identity, then $g$ has no fixed point besides $\z$ and the points of $A$. Therefore, a non-identical perspectivity has a unique axis and a unique center. A $(\z,A)-$perspectivity  is called an {\em elation} or a {\em homology} depending on whether $\z\in A$ or $\z\notin A$. The non-identical elations with axis $A$ are just the perspectivities without any fixed point outside $A$. Therefore they form, togheter with the identity, a normal subgroup of the full  collineation group   of the projective space fixing $A$. This elation group acts fix-point-free (that is, semiregularly) outside $A$. In particular the set of all $(\z,A)-$perspectivities is a subgroup of this latter group.

If $g$ is any perspectivity, $\langle g\rangle$ induces a semiregular permutation group on the non-fixed points of any line not contained in the axis through the center. If $k$ is the order of $g$, then $g$ is either an elation or a homology if  either $k|p^h$ or $k|p^h-1$, respectively. Thus, any $(\z,A)-$elation group is an elementary abelian $p-$group.

In this paper, we are interested in the geometry of the point-orbits of a $(\z,A)-$elation group of $\PG(r-1,p^h)$. The motivation for us to study this problem is that this type of elation groups play an important role in finite geometries as they appear as automorphism groups of relevant objects such as unitals, blocking sets and maximal arcs in projective spaces.


\section{The geometry of  $(\z,A)-$elation groups} \label{sec_1}

Let $E$ denote a $(\z,A)-$elation group in $\PG(r-1,p^h)$. To study the geometry of the orbits on the non-fixed points of $E$  it is more convenient to introduce  homogeneous projective coordinates  $(X_0,\ldots,X_{r-1})$ of $\PG(r-1,p^h)$ in such a way that $\z=(0,0,\ldots,0,1)$ and $A$ has equation $X_0=0$. For a fixed basis of the underlying vector space $V(r,p^h)$, the elements of  $E$ have representation by a lower triangular matrix
\[
\M_\lambda=\begin{pmatrix}
1 & 0 &\ldots & 0 & 0\\
0 & 1 & \ldots & 0 & 0\\
\vdots &\vdots&\ddots &\vdots &\vdots\\
\lambda &	0 &\ldots&0&1	
\end{pmatrix}
\]
for some $\lambda\in\GF(p^h)$; we will denote by $e_\lambda$ the element of $E$ with associated matrix $\M_\lambda$.  As the group operation of $E$ is defined by $e_\lambda*e_{\lambda'}=e_{\lambda+\lambda'}$, we can identify  $E$ with the  additive subgroup $H_E=\{\lambda:e_\lambda\in E\}$ of $\GF(p^h)$. We say that two additive subgroups of $\GF(p^h)$ are {\em equivalent} if they differ by a non-zero scalar in $\GF(p^h)$. We say that two $(\z,A)-$elation groups of $\PG(r-1,p^h)$ are {\em equivalent} if the corresponding additive subgroups are equivalent.

\begin{lemma} \label{lem_1}
In $\PG(r-1,p^h)$ two $(\z,A)-$ elation groups are conjugate in $\PGL(r,p^h)$ if and only if they are equivalent.
\end{lemma}
\begin{proof}
Let $E$ and $\bar E$ be two conjugate $(\z,A)-$elation groups, i.e $gEg^{-1}=\bar E$,  for some $g\in\PGL(r,p^h)$. Let
\[
\G=\begin{pmatrix}
a_{11} & \a_1 & a_{1r}\\
\a_{2}' & \A & \a_{3}'\\
a_{r 1} &	\a_{4} &a_{r r}
\end{pmatrix}
\]
be a matrix representation of $g$ with $a_{11},a_{rr}\in \GF(p^h)\setminus\{0\}$, $a_{1r},a_{r1}\in\GF(p^h)$, $\a_{i}\in\GF(p^h)^{r-1}$ and $\A$ is a $(r-1)\times(r-1)$ matrix with entries in $\GF(p^h)$; here $\a'$ denote the traspose of $\a$. The equality $gE=\bar Eg$ in $\PGL(r,p^h)$ implies that for every $e_\lambda\in E$ there exist $e_{\bar\lambda}$ and $\rho\neq0$ such that $\G\M_\lambda=\rho\M_{\bar\lambda}\G$. A straightforward computation shows that   $\a_1=\a_3=\0, a_{1 r}=a_{r1}=0$ and $H_{\bar E}=a_{rr}\cdot a_{11}^{-1} H_E$. As $a_{11}\neq 0$ we can set $a_{11}=1$ so that $H_{\bar E}=a_{rr} H_E$ i.e., $E$ and $\bar E$ are  equivalent. 

Conversely,  if $E$ and $\bar E$ are equivalent, that is $H_{\bar E}=\alpha H_E$ for some $\alpha\in\GF(p^h)\setminus\{0\}$, then $gEg^{-1}=\bar E$ for every element $g$ of $\PGL(r,p^h)$ having matrix representation  
\[
\begin{pmatrix}
1 & 0 & \ldots & 0 & 0\\
a_{21} & a_{22} & \ldots & 0 & 0\\
\vdots &\vdots& & \ddots &\vdots\\
a_{r 1} &	a_{r 2} &\ldots&a_{r r-1}&\alpha
\end{pmatrix}.
\]
\end{proof}

Let $E$ be a $(\z,A)-$elation group of order $p^m$, $1\le m\le h$, of $\PG(r-1,p^h)$. Since we can identify $E$ with the associated additive subgroup $H_E$ of $\GF(p^h)$, it turns out that  $E$ is an $m-$dimensional $\GF(p)-$vector space.  Let $\GF(p^n)$ be a subfield of $\GF(p^h)$ such that we can consider $E$ as a $\GF(p^n)-$vector space. Then we will say that $E$ has {\em  dimension} $d=m/n$  over $\GF(p^n)$ and we write $\dim_{p^n}(E)=d$. If $\GF(p^n)$ is maximal with respect to the above property then we say that $d=m/n$ is the {\em minimal} dimension of $E$.

If $E$ has dimension $d$ over $\GF(p^n)$,  we set $d'=h/n$. Then, it becomes quite natural to represent $\PG(r-1,p^h)$ into $\PG((r-1)d',p^n)$ by using the well-known  {\em Andr\'e/Bruck-Bose representation} which we briefly recall.

A $(t-1)-${\em spread} of $\PG(s-1,p^n)$ is a set $\cS$ of $(t-1)$-dimensional subspaces such that every point of $\PG(s-1,p^n)$ is contained in exactly one element of $\cS$. A $(t-1)-$spread exists if and only if $t|s$.  We will say that a $(t-1)-$spread $\cS$ of $\PG(s-1,p^n)$ fills a $(m-1)-$dimensional subspace $\Lambda$ if any element of $\cS$ is either contained in $\Lambda$ or disjoint from $\Lambda$. A $(t-1)-$spread is said to be {\em Desarguesian } (or {\em normal}) if, taking as $(m-1)-$dimensional subspaces the set of $(mt-1)-$dimensional subspaces filled by $\cS$, for $m=1,2,\dots,r/t$, and the inclusion inherited from $\PG(s-1,p^n)$ gives a projective space $\PG(s/t-1,p^{nt})$.

 In $\PG((r-1)d',p^n)$ we choose homogeneous projective coordinates 
 \[(X_{00},X_{11},\ldots,X_{1d'},\ldots,X_{21},\ldots,X_{2d'},\ldots,X_{r-11},\ldots,X_{r-1d'}).\]
Let $\x=(1,x_1,\ldots,x_{r-1})$ be an affine point of $\PG(r-1,p^h)$. If the element $x_i$ of $\GF(p^h)$ is replaced by its corresponding $d'$-ple $(x_{i1},\ldots,x_{id'})$ over $\GF(p^n)$, then $\x$ defines the affine point $\x^*=(1,x_{11},\ldots,x_{1d'},\ldots,x_{r-11},\ldots,x_{r-1d'})$ of $\PG((r-1)d',p^n)$, and conversely.

Now we extend  the above map  to a map $*$ from $\PG(r-1,p^h)$ into $\PG((r-1)d',p^n)$ by taking points of the hyperplane  $A$  into elements of a Desarguesian $(d'-1)-$spread $\cS$ of the hyperplane $A^*$ of $\PG((r-1)d',p^n)$.  Then it is possible to define a point-line geometry $\Pi(\cS)$ in the following way: (i) the {\em points} are the points of $\PG((r-1)d',p^n)\setminus A^*$ ({\em affine points}) and the elements of $\cS$ ({\em point at infinity}), (ii) the {\em lines} are the $d'-$dimensional subspaces of $\PG((r-1)d',p^n)$ intersecting $A^*$ in an element of $\cS$ ({\em affine lines}) and the $(2d'-1)-$dimensional subspaces generated by two distinct elements of $\cS$ ({\em lines at infinity}), (iii) the {\em point-line incidences} are inherited from $\PG(r-1,p^h)$. It turns out that the incidence structure $\Pi(\cS)$ is isomorphic to $\PG(r-1,p^h)$. 

In the following $A^*$ will be the hyperplane at infinity of $\PG((r-1)d',p^n)$. It is clear that the center $\z$ defines the $(d'-1)-$dimensional subspace $\z^*$ of $A^*$. We choose a basis of the underlying vector space $V((r-1)d'+1,p^n)$ of $\PG((r-1)d',p^n)$ such that  $A^*:X_{00}=0$, $\z^*:X_{00}=X_{11}=\ldots=X_{1d'}=\ldots=X_{r-21}=\ldots=X_{r-2d'}=0$.

\begin{lemma}
Let  $E$ be a $(\z,A)-$elation group of order $p^m$ and dimension $d=m/n$ of $\PG(r-1,p^h)$. Under the correspondence ${}^*$ the orbit $\x^E$  of the affine point $\x$ of $\PG(r-1,p^h)$ defines a $d-$dimensional affine subspace of $\PG((r-1)h/n,p^n)$. On the hyperplane $A^*$, $E$ induces the identity.
\end{lemma}
\begin{proof}
Set $d'=h/n$. The $E-$orbit of the point $\x=(1,x_1,\ldots,x_{r-1})$ is $\x^E=\{(1,x_1,\ldots,x_{r-1}+\lambda):\lambda\in H_E\}$.  If $(\lambda_1,\ldots,\lambda_{d'})$ is the $d'-$ple  over $\GF(p^n)$ associated with $\lambda$ then $\x^E$ defines the set of points $(\x^E)^*=\{\y^*=(1,x_{11},\ldots,x_{1d'},\ldots,x_{r-11}+\lambda_1,\ldots,x_{r-1d'}+\lambda_{d'})):\y\in \x^E\}$. Since $H_E$ is a $d-$dimensional vector spaces over $\GF(p^n)$  we get that  $(\x^E)^*$ is a $d-$dimensional affine subspace of $\PG((r-1)d',p^n)$.

It is easy to check that $E$ fixes pointwise $A^*$.
\end{proof}
\begin{lemma}\label{lem_2}
All $d-$dimensional subspaces $(\x^E)^*$, with $\x$ an affine point of $\PG(r-1,p^h)$, intersect $\z^*$ in a common  $(d-1)-$dimensional subspace.
\end{lemma}

\begin{proof}
Let $\x=(1,x_1,\ldots,x_{r-1})$ be an affine point of $\PG(r-1,p^h)$ and let $\y\in\x^E$. Then $\y=(1,x_1,\ldots,x_{r-1}+\lambda)$ for some $\lambda\in H_E$. The line $\langle \x^*,\y^*\rangle$ of $\PG((r-1)d',p^n)$ joining $\x^*$ and $\y^*$ meets $A^*$ in $(0,0,\ldots,0,\lambda_1,\ldots,\lambda_{d'})$ which is a point of $\z^*$. We see at once that $\langle \x^*,\y^*\rangle\cap \z^*$ does not depends neither on the choice of the affine point $\x$ of $\PG(r-1,p^h)$ nor on that of $\y\in\x^E$. As $\lambda$ varies in $H_E$, the points  $(0,0,\ldots,0,\lambda_1,\ldots,\lambda_{d'})$ are precisely the points of  a $(d-1)-$dimensional subspace of $\z^*$.
\end{proof}

Let $\cH$ be a $(d-1)-$dimensional subspace of $\z^*$ with equations (in $\z^*$)
\begin{equation}\label{eq_1}
\A\X_{r-1}=\0,
\end{equation}
 where $\A\in\GL(d',p^n)$ has rank $d'-d$ and $\X_{r-1}=(X_{r-11},\ldots,X_{r-1d'})^{t}$. 
 
The solutions of (\ref{eq_1}) form the underlying vector subspace $H$ of $\cH$. For every $d'-$ple $\mathbf\lambda=(\lambda_1,\ldots,\lambda_{d'})$ in $H$ consider its corresponding element $\lambda$ of $\GF(p^h)$. Let $E$ be the $(\z,A)-$elation group $\{e_\lambda:\mathbf\lambda\in H\}$ of $\PG(r-1,p^n)$. Then it is easy to check that, for every affine point $\x$ of $\PG(r-1,p^h)$,  the $d-$dimensional subspace $\langle\x^*,\cH\rangle$ of $\PG((r-1)d',p^n)$ is precisely $(\x^E)^*$.

Thus we have proved the following result.

\begin{proposition}\label{rem_1}
There exists a 1-1 correspondence from $(\z,A)-$elation groups of order $p^m$ and dimension $d=m/n$ of $\PG(r-1,p^h)$ and $(d-1)-$dimensional subspaces  contained in  $\z^*$.
\end{proposition}
%
%

\section {Elation groups of $\PG(r-1,p^h)$ and orbits of Singer groups} \label{sec_2}

 Let $V=\langle v_1,\ldots,v_r\rangle$ be a vector space of dimension $r$ over $\GF(p^h)$. 

Here we review some results on Singer groups of vector and projective spaces.
 
Denote by $\frob$ the Frobenius automorphism of $\GF(p^{hr})$ over $\GF(p^h)$. Let $\mu$ be a generator of the cyclic group $\GF(p^{hr})^*$ and $m(X)=X^r+a_{r-1}X^{r-1}+\ldots+a_1X+a_0$ be its minimal polynomial over $\GF(p^h)$. The companion matrix $C$ of $m(X)$ is
\[
\begin{pmatrix}
0 & 0 &\ldots & 0 & -a_0\\
1 & 0 &\ldots & 0  & -a_1\\
0 & 1 & \ldots & 0 & -a_2\\
\vdots &\vdots&\ddots & &\\
0 &	0 &\ldots&1&-a_{r-1}	
\end{pmatrix}.
\]

The linear transformation $\sigma:V\longrightarrow V$ whose matrix with respect to the basis $v_1,\ldots,v_r$ is $C$ is nonsingular and the cyclic group $S=\langle\sigma\rangle$ generated by $\sigma$ has order $p^{hr}-1$ and acts transitively on the nonzero vectors of $V$. The group $G$ s called a {\em Singer cyclic group} of $V$; see, e.g., \cite{hu},\cite{s}.

On finite projective space $\PG(r-1,p^h)$,  the Singer cycle $\sigma$ induces  a linear collineation which we still denote by $\sigma$ and the group $S$ acts regularly on points (and hyperplanes) of the space. Such a group is called a {\em Singer group} of $\PG(r-1,p^h)$. 

In the literature, the actions of Singer groups are usually considered on points
and lines. In \cite{dgg}, the geometric structure of Singer line orbits in $\PG(3, q)$ were investigated. This was extended to Singer plane orbits also in higher dimensional
projective spaces; see \cite{ems}.

In his paper \cite{kd} Drudge has studied the action of the Singer group on the $(t-1)$-dimensional subspaces of $\PG(r-1,p^n)$ for arbitrary $t$ and $r$; he determined the possible orbit sizes and how many orbits of each type occur.

{}For the convenience of the reader we repeat the relevant material from \cite{kd} without proofs, thus making our exposition self-contained.

A $(t-1,k)-${\em cover} of $\PG(r-1,p^h)$ is a set $\cK$ of $(t-1)-$subspaces such that every point of $\PG(r-1,p^h)$ is on exactly $k$ elements of $\cK$; so a $(t-1)-$spread is a $(t-1,1)-$cover. The $(t-1,k)-$covers generalize to subspaces of arbitrary dimension the $k-$covers of lines introduced by Ebert in \cite{e}.

For any $t\ge1$, $\theta_t$ will denote the numbers of points of $\PG(t-1,p^h)$, that is
$\theta_t=p^{h(t-1)}+p^{h(t-2)}+\ldots+p^h+1.$

\begin{proposition} \cite{kd}
\begin{itemize}
\item[i.] A Singer group $S$ of $\PG(r-1,p^h)$ has an orbit which is a $(t-1)-$spread if and only if $t|r$. In this case there is exactly one such orbit, say  $\cS$. The $S-$stabilizer of any $X\in\cS$ is $Stab_S(X)=\langle\sigma^{\theta_r/\theta_t}\rangle$. $\cS$ is a Desarguesian $(t-1)-$spread and the factor group $S/Stab_S(X)$ is a Singer group of the corresponding $\PG(r/t-1,p^{ht})$.
 \item[ii.] The $S-$stabilizer of a $(t-1)-$dimensional subspace has size $\theta_u$, and is therefore $\langle\sigma^{\theta_r/\theta_u}\rangle$ for some $u|(t,r)$. In other words, each orbit of $S$ on $(t-1)-$subspaces is a $(t-1,\theta_t/\theta_u)-$cover of $\PG(r-1,p^h)$ for some $u|(t,r)$.
\end{itemize}
\end{proposition}
\begin{remark} \label{rem_2}
Suppose $t|r$ and let $\cS$ be the unique $(t-1)-$spread orbit of $S$. Let $\cK$ be a $(t-1,\theta_r/\theta_u)-$cover of $\PG(r-1,p^h)$ for some $u|(t,r)$. Then each element of $\cS$ intersects elements of $\cK$ in a $\PG(u-1,p^h)$. Such elements form a Desarguesian $(u-1)-$spread which fills every element of $\cK$. Further, this orbit is isomorphic to $\PG(r/u-1,p^{hu})$. 
\end{remark}

In the remainder of this section we study the relation between equivalent $(\z,A)-$elation groups and orbits of Singer groups of projective geometries.

Let $E$ and $\bar E$ be two equivalent $(\z,A)-$elation groups of $\PG(r-1,p^h)$. Then  $H_{\bar E}=\alpha H_E$ for some non-zero element $\alpha\in\GF(p^h)$. Assume that $E$ has order $p^m$ and dimension $d=m/n$ over $\GF(p^n)$ i.e., $H_E$ is a $d-$dimensional vector space over $\GF(p^n)$. From Lemma \ref{lem_1} it is clear that $\bar E$ has order $p^m$. Furthermore, it easy to check that also $\alpha H_E$ turns out to be a $d-$dimensional vector space over $\GF(p^n)$ which implies that $\bar E$ has dimension $d$ over $\GF(p^n)$.

In the proof of Lemma \ref{lem_2} we have seen that the projective subspaces defined by $H_E$ and $H_{\bar E}$, which we denote by $\overline H_E$ and $\overline H_{\bar E}$, are precisely the intersections of $(\x^E)^*$ and  $(\x^{\bar E})^*$ with $\z^*$, respectively. Since the multiplicative group of $\GF(p^h)$ over $\GF(p^n)$ is cyclic, multiplication in $\GF(p^h)$ over $\GF(p^n)$ defines a Singer cyclic group $S$ in $\z^*=\PG(h/n-1,p^n)=\PG(d'-1,p^n)$.  This implies that $\overline H_E$ and $\overline H_{\bar E}$ lie  in the same $S-$orbit. 
In particular, from  the definition of minimal dimension of an elation group and Remark \ref{rem_2}, we see that  if $d$ is the minimal dimension of $E$ then such an orbit is a $(d-1,\theta_{d})-$cover of $\z$. We summarize this discussion.

\begin{theorem}\label{th_1}
There exists a 1-1 correspondence from  conjugacy classes of $(\z,A)-$elation groups of order $p^m$ and dimension $d=m/n$ of $\PG(r-1,p^h)$ into  the orbits of a Singer group of  $\PG(h/n-1,p^n)$ on $(d-1)-$dimensional subspaces.
This correspondence induces a 1-1 correspondence from  conjugacy classes of  $(\z,A)-$elation groups of $\PG(r-1,p^h)$ of minimal dimension $d=m/n$ into  orbits of a Singer group of  $\PG(h/n-1,p^n)$ which are $(d-1,\theta_{d})-$covers. 
\end{theorem}
\begin{corollary}
The number of $(\z,A)-$elation groups in $\PG(r-1,p^h)$ of order $p^m$ and dimension $d=m/n$  is 
\begin{equation}\label{eq_2}
\frac{1}{\theta_{h/n}}\sum_{t|(d,h/n)}{\left({{h/(nt)} \brack {d/t}}_{p^{nt}}\sum_{u|t}{\mu(t/u)\theta_u}\right)};
\end{equation}
here $\mu$ denotes the M\"obius function \cite{a} and ${a \brack b}_{p^{k}}$ the number of $(b-1)-$dimensional subspaces of $\PG(a-1,p^k)$.

In particular, there are exactly 
\begin{equation}\label{eq_3}
\frac{1}{\theta_{h/n}}\sum_{t|(d,h/n)}{\mu(t){{h/(nt)} \brack {d/t}}_{p^{nt}}}
\end{equation}
$(\z,A)-$elation groups in $\PG(r-1,p^h)$ of order $p^m$ and minimal dimension $d$.
\end{corollary}
\begin{proof}
In  \cite[Theorem 2.1]{kd} the number of orbits of a Singer group of $\PG(h/n-1,p^n)$ on the set of $(d-1)-$dimensional subspaces was determined and it is precisely given by the formula (\ref{eq_2}). 

Putting $u=1$ in the formula (2) of Theorem 2.1 in \cite{kd} one gets the number of orbits of a Singer group of $\PG(h/n-1,p^n)$ on the set of $(d-1)-$dimensional subspaces which are $(d-1,\theta_{d})-$covers, that is
\[
\frac{1}{\theta_{h/n}}\sum_{t|(d,h/n)}{\mu(t){{h/(nt)} \brack {d/t}}_{p^{nt}}}.
\]

The results are then an immediate consequence of Theorem \ref{th_1}.
\end{proof}

\section{Concluding remarks}
 
Elation groups play a central role in finite geometries as they can appear as automorphism groups of relevant objects such as unitals, blocking sets and maximal arcs in projective planes. 

While describing the geometry of the orbits of $(\z,A)-$elation groups of a projective geometry our feeling is that our results may be useful to improve the knowledge (from a group-theoretic point of view) of geometric objects with such a automorphism group.

With the notation introduced in the previous sections, let $E$ be a $(\z,A)-$elation group of order $p^m$ and  minimal dimension $d=m/n$ of $\PG(r-1,p^h)$ and let $\cB$ be a set of points which is invariant under $E$. It is evident that the geometry of the representation $\cB^*$ of $\cB$ depends on the intersection behaviour  of $\cB$ with the center $\z$ and the axis $A$. 

 In the case that $\z$ is a point of $\cB$  and $A$ is a tangent hyperplane  of $\cB$ at $\z$ we see that $\cB^*$ is a cone with vertex a $(d-1)-$dimensional subspace which belongs to a $(d-1,\theta_d)-$cover of $\z^*$. It turns out that $\cB$ is union of $\GF(p^n)-$linear sets; see \cite{op}. 

If $\cB$ is a unital in $\PG(2,p^{2h})$ i.e, a set of $p^{3h}+1$ points  such that each line meets $\cB$ in 1 or $p^h+1$ points, with  a $(\z,A)-$elation automorphism group $E$, then it easy to see that $\z$ must be a point of $\cB$ and $A$ is a tangent. In fact, if we assume that $\z\notin \cB$ then there is a line $\ell$ on $\z$ fixed by $E$ and intersecting $\cB$ in $p^h+1$ points. Since $E$ also fixes $\cB$ then $E$ must fix $\ell\cap\cB$. This implies that $|E|=p^m$ should divide $p^h+1$, a contradiction. Then $\z\in\cB$. Assume now that the axis $A$ meets $\cB$ in $p^h+1$ points and let $P$ be a point of $A\cap\cB$ but not $\z$. It is well known that on $P$ there is a unique  tangent $t$ of $\cB$. But $t$ is not fixed by $E$ (because $t$ is not on $\z$), contrary to the uniqueness  of $t$. We conclude that $A$ must be the tangent at $\z$ of $\cB$.

Examples of such unitals are the Buekenhout-Metz unitals  which are represented in $\PG(4,p^{h})$  by cones projecting  an ovoid  in a suitable 3-dimensional subspace from a point of $\z^*$  not contained in the 3-dimensional subspace \cite{b} \cite{m}. In $\PG(2,p^{h})$, the latter unitals are invariant under  an elation group of order $p^h$ and minimal dimension 1 that is,  they are union of $p^{h}$ Baer sublines through $\z$. It should also be noted that in the papers  \cite{cq} and \cite{ckp} it was proved that if $\cB$ is a unital of $\PG(2,p^{h})$, where $p^{h/2}>2$, such that each of the $p^{h}$ secant lines through $\z$ meets $U$ in a Baer subline, then $\cB$ is a Buekenhout-Metz unital. 

To our knowledge not much is known for maximal arcs and blocking sets with the same invariance property.

Generalizing the Buekenhout-Metz construction,  Mazzocca and Polverino constructed minimal blocking sets in $\PG(2,p^{hn})$ from cones of $\PG(2n,p^h)$ \cite{mp}. Such cones have an $s-$dimensional subspace $\Omega$ of $\z^*$, with $0\le s\le n-2$, and generators $(s+1)-$dimensional affine subspaces through $\Omega$. It turns out that such blocking sets are left invariant by a $(\z,A)-$elation group of order $p^{h(s+1)}$ and dimension $s+1$.  

Something can  also be said about  $(k,n)-$maximal arcs i.e.,  sets of $k$ points of $\PG(2,p^h)$ meeting every line in just $n$ points or in none at all \cite{h}. If $\cK$ is a non-trivial maximal arc i.e., $\cK$ is neither a single point nor an affine plane, it is well known that $|\cK|=p^hn-p^h+n$. In particular, if $\alpha$ denotes the constant number of secant lines on a point not in $\cK$, we have $(p^h+1-\alpha)n=p^h$ with $\alpha>1$. Then  $n$ is a power of $p$ and $(\alpha,p)=1$.

The question of existence of maximal arcs in $\PG(2,p^h)$ has a negative answer for $p$ odd, while it has been settled in the affirmative for $p=2$; see   \cite{bbm}, \cite{de}, \cite{rm}.

\begin{proposition}
Let $\cK$ be a maximal $(k,n)-$arc of $\PG(2,2^h)$ which is invariant under a $(\z,A)-$elation group $E$ of order $2^m$, $1\le m\le h-1$. Then $\z\not\in\cK$ and $A\cap\cK\neq\emptyset$.
\end{proposition}
\begin{proof}
First we notice that $n=2^{m+l}$, with $0\le l\le h-m-1$. Hence $\alpha=2^h+1-2^{h-m-l}$.

Assume $\z\in\cK$. Then $A$ must be a secant of $\cK$. Fix a line $\ell$ on $\z$. Then $\ell$ meets $\cK$ in $n-1$ points other then $\z$. Since $\ell\cap\cK$ is union of orbits of $E$ we must have $2|n-1$, a contradiction. We conclude that $\z\not\in\cK$.

Let $A$ be skew with $\cK$. Let $P$ be a point on $A\setminus\{\z\}$. As each secant line on $P$ meets $\cK$ in $n$ points we get $\alpha=2^h+1-2^h/n\ge n$. It follows easily that if $\alpha=n$, then $n\in\{1,2^h\}$ which implies that $\cK$ is a trivial  maximal arc. Then $\alpha>n$. The group $E$ splits the $\alpha$ lines on $P$ into orbits of length $|E|=2^m$ which implies $2^m|\alpha$, a contradiction.
\end{proof}

In the paper \cite{hp}, stabilizers of some known maximal arcs in Desarguesian projective planes are determined. In particular,  we see  from Theorem 2.3 of \cite{hp} that Denniston maximal arcs \cite{de} have not elations if $n>2$ while Thas 1974 maximal arcs \cite{hp} have no elations by construction. On the other hand, when $n=2$ translation hyperovals are stabilized by an involutory elation \cite[Theorem 4.2]{hp}.

\noindent
{\bf Author's Address}

\noindent
Alessandro Siciliano {\em (corresponding author)}\\[.03in]
Dipartimento di Matematica ed Informatica\\ [.03in]
Universit\`a degli Studi della Basilicata\\[.03in]
Via dell'Ateneo Lucano, I-85100 Potenza , Italy\\[.03in]
{\em e-mail:} alessandro.siciliano@unibas.it

\bigskip
\noindent
Nicola Durante \\[.03in]
Dipartimento di Matematica ed Applicazioni\\ [.03in]
Universit\`a degli Studi di Napoli "Federico II"\\[.03in]
Complesso di Monte S. Angelo - EdiÞcio T \\[.03in]
via Cintia, I-80126 Napoli, Italy. \\[.03in]
{\em e-mail:} ndurante@unina.it


\begin{thebibliography}{10}

%
\bibitem{a} T.M. Apostol, {\em Introduction to analytic number theory}, Spriger-Verlag, New York, 1976. 
%
%
\bibitem{bbm} S. Ball, A. Blokhuis and F. Mazzocca, Maximal arcs in Desarguesian planes of odd order do not exist, {\em Combinatorica} {\bf 17} (1997), 31--41. 
%
\bibitem{b} F. Buekenhout, Existence of unitals in finite translation planes of order $q^{2}$ with a kernel of order $q$, {\em Geometriae Dedicata} {\bf 5} (1976),  189--194.
%
%
\bibitem{cq} R. Casse and  C.T. Quinn, Concerning a characterisation of Buekenhout-Metz unitals, {\em J. Geom.}  {\bf 52}  (1995), 159--167. 
%
\bibitem{ckp} L.R.A. Casse, C.M. O'Keefe and  T. Penttila, Characterizations of Buekenhout-Metz unitals,  {\em Geom. Dedicata}  {\bf 59}  (1996),   29--42.
%
%
%
\bibitem{de} R.H.F. Denniston,  Some maximal arcs in finite projective planes,  {\em J. Combinatorial Theory} {\bf 6} 1969, 317--319.
%
\bibitem{kd} K. Drudge, On the orbits of Singer groups and their subgroups, {\em Electron. J. Combin.}  {\bf 9}  (2002), 10 pp. (electronic).
%
\bibitem{e} G. Ebert, The completion problem for partial packings, {\em Geom. Dedicata}   {\bf 18}  (1985), 261--267. 
%
\bibitem{ems} G. Ebert, K. Metsch and T. Sz\"onyi, {\em Geom. Dedicata} {\bf 70} (1998), 181--196
%
\bibitem{dgg} D.G. Glynn, On a set of lines of $\PG(3,q)$ corresponding to a maximal cap contained in the Klein quadric of $\PG(5,q)$, {\em Geom. Dedicata} {\bf 26} (1988), 273--280. 
%
\bibitem{hp} N. Hamilton and T. Penttila, Groups of maximal arcs, {\em J. Combin. Theory Ser. A} {\bf 94} (2001),  63--86. 
%
\bibitem{h} J. W. P. Hirschfeld, Projective geometries over finite fields, 
Oxford University Press, New York, 1998.
%
\bibitem{hu} B. Huppert, {\em Endliche Gruppen I}, Spriger, Berlin, 1967. 
%
\bibitem{rm} R. Mathon, New maximal arcs in Desarguesian planes, {\em J. Combin. Theory Ser. A} {\bf 97} (2002),  353--368.
%
\bibitem{mp} F. Mazzocca and O. Polverino, Blocking sets in ${\PG}(2,q^n)$ from cones of ${\PG}(2n,q)$, {\em J. Algebraic Combin.} {\bf 24} (2006),  61--81. 
%
\bibitem{m} R. Metz, On a class of unitals, {\em Geom. Dedicata} {\bf 8} (1979),  125--126. 
%
\bibitem{op} O. Polverino, Linear sets in finite projective spaces, {\em Discrete Math.}, {\bf 310} (2010), 3096--3107.
%
\bibitem{s} J. Singer, A theorem in finite projective geometry and some applications to number theory, {\em Trans. Am. Math. Soc.}  {\bf 43}  (1938), 377--385.
%
%
\end{thebibliography}
\end{document}